\documentclass[11pt,reqno]{amsart}
\usepackage{amsmath}
\usepackage{cases}
\usepackage{mathrsfs}
\usepackage{bbm}
\usepackage{amssymb}
\usepackage{amscd}
\usepackage{amsfonts,latexsym,amsmath,
amsthm,amsxtra,mathdots,amssymb,latexsym,mathabx}
\usepackage[all,cmtip]{xy}
\RequirePackage{amsmath} \RequirePackage{amssymb}
\usepackage{color}
\usepackage{colordvi}
\usepackage{multicol}
\usepackage{hyperref}
\usepackage{mathtools}
\usepackage[top=1.1in,bottom=1.1in,left=1.1in,right=1.1in]{geometry}   
\usepackage{xcolor}

\hypersetup{
    colorlinks,
    linkcolor={red!100!black},
    citecolor={blue!100!black},
    urlcolor={blue!100!black}
}
\usepackage{cite}

\newcommand{\bea}{\begin{eqnarray}}
\newcommand{\eea}{\end{eqnarray}}
\newcommand{\bna}{\begin{eqnarray*}}
\newcommand{\ena}{\end{eqnarray*}}

\numberwithin{equation}{section}

\theoremstyle{plain}
\newtheorem{lemma}{Lemma}[section]
\newtheorem{theorem}[lemma]{Theorem}

\theoremstyle{definition}
\newtheorem{definition}[lemma]{Definition}
\newtheorem{remark}{Remark}

\renewcommand{\Re}{\operatorname{Re}}

\newcommand{\GL}{\operatorname{GL}}

\newcommand{\dd}{\mathrm{d}}

\newcommand{\Kl}{\mathrm{Kl}}
\newcommand{\ord}{\mathrm{ord}}

\begin{document}

\title[Triple divisor type functions in arithmetic progressions]
{Triple divisor type functions in arithmetic progressions to prime power moduli}

\author{Hui wang}
\address{The University of Hong Kong Shenzhen Institute of Research and Innovation, 
  Shenzhen, Guangdong 518052, China}
\email{wh0315@mail.sdu.edu.cn}

\author{Tengyou Zhu}
\address{Data Science Institute, Shandong University
                       \\Jinan, Shandong 250100, China}
\email{tengyou.zhu@sdu.edu.cn}

\begin{abstract}
We study the equidistribution in arithmetic progressions modulo powers of a fixed odd prime of Hecke eigenvalues of noncuspidal $\mathrm{GL}_3$ 
automorphic forms, 
with a focus on the convolution coefficients $\lambda_{1\boxplus f}(n)=(\lambda_f\star 1)(n)$,
attached to a level $1$ holomorphic primitive cusp form $f$. 
Utilizing techniques introduced by Kowalski--Lin--Michel and Mili\'{c}evi\'{c}, 
we prove the exponent $\vartheta=1/2+7/458-\varepsilon$ is admissible.
\end{abstract}

\maketitle

\section{Introduction}
The equidistribution of arithmetic functions over arithmetic progressions constitutes one of the central, 
long-standing threads of analytic number theory. 
A core object of inquiry here is the so-called \emph{level of distribution}: given an arithmetic function
$g:\mathbb{N}\rightarrow\mathbb{C}$,  we seek the maximal exponent $\vartheta>0$ such that
for any $q\leq X^{\vartheta}$, any coprime residue $a\bmod q$ and arbitrary constant $A>0$, the uniform bound
\begin{equation}\label{level-of-distribution}
\mathop{\sum_{n\leq X}}_{n\equiv a \bmod q}g(n)-\frac{1}{\varphi(q)}\mathop{\sum_{n\leq X}}_{(n,q)=1}g(n)
\ll_{g,A} \frac{X}{q}\left(\log X\right)^{-A}
\end{equation}
holds, where $\varphi$ is the Euler totient function and the implied constant depends only on $g$, $A$.
The conjecture predicts that $\vartheta \!=\!1-\varepsilon$ is attainable for 
many classical arithmetic functions, yet crossing critical thresholds like $\vartheta=1/2$ 
remains a formidable technical barrier tied to sharp estimates of multi-dimensional exponential sums,
especially hyper-Kloosterman sums attached to higher-rank automorphic representations.

Classical prototypes of this problem begin with the von Mangoldt function 
$\Lambda(n)$, encoding prime distribution.
The classical Siegel--Walfisz theorem implies that \eqref{level-of-distribution}
holds as $q\leq (\log X)^{B}$ for some constant $B$, 
whereas the generalized Riemann hypothesis predicts $q\leq X^{1/2-\varepsilon}$. 
The significant result is that the Bombieri--Vinogradov theorem
confirms this prediction on average over the moduli (see Montgomery--Vaughan\cite{MV}).

For an integer \(k\geq2\), let \(\tau_k\) denote the \(k\)-fold divisor function
$$
\tau_k(n):=\#\{(n_1, n_2, \cdots, n_k)\in \mathbb{N}^k: n_1n_2\cdots n_k=n\}.
$$
The conjecture that \eqref{level-of-distribution} holds for $\vartheta=1-\varepsilon$, if correct,
has profound implications for our understanding of
the distribution of primes
in arithmetic progressions to very large moduli,
going far beyond the direct reach of
the generalized Riemann hypothesis.
For $k=2$, the best known value that \eqref{level-of-distribution} can hold is $\vartheta=2/3-\varepsilon$,
which stems from work of Selberg (unpublished), Hooley \cite{Hooley} and Heath-Brown \cite{HB1979}.
In each case Weil's bound for the Kloosterman sum is crucial in the proof.
For $k\geq 4$, classical conclusions can be found in Lavrik \cite{Lavrik},
Smith \cite{Smith} and some references therein.
A landmark breakthrough arrived for the triple divisor function $\tau_3(n)$, 
the Hecke eigenvalue of the trivial isobaric representation $1\boxplus1\boxplus 1$ on $\rm GL_3$.
In this case, Friedlander--Iwaniec~\cite{FI1985} first broke the $1/2$ barrier with
the exponent of distribution $\vartheta=1/2+1/230$ for
general moduli. This was subsequently improved by Heath-Brown~\cite{HB1986}, 
Fouvry--Kowalski--Michel~\cite{FKM2015} (for prime moduli), and Xi \cite{Xi} (for smooth moduli).
More recently,
Sharma \cite{Sharma} deployed the delta symbol approach to push the exponent to
 $\vartheta=1/2+1/30-\varepsilon$
over square-free and prime power moduli, relying on refined bilinear sum bounds for $\textrm{GL}_2$
weighted Kloosterman sums.

Beyond classical divisor functions, the theory of automorphic $L$-functions motivates parallel 
distribution questions for Hecke eigenvalues of a general $\textrm{GL}_d$ automorphic representation $\pi$,
with associated Dirichlet series
\bna
L(\pi, s)=\sum_{n\geq 1}\frac{\lambda_{\pi}(n)}{n^s}=\prod_pL(\pi_p,s), \quad\Re(s)>1.
\ena
In general, when $\pi$ is a $\textrm{GL}_d$ automorphic form,
a standard analytic tool for such congruence sums $n\equiv a \bmod q$ combines additive character orthogonality 
with the Vorono\"{i} summation formula,
transforming the original sum into a dual sum of approximate length $q^d/X$, involving the dual Hecke eigenvalues
$\overline{\lambda_{\pi}(n)}$, together with a twist by a $(d-1)$-dimensional
hyper-Kloosterman sum $\text{Kl}_d(a;q)$, defined by
$$
\Kl_d(a;q)=\frac{1}{q^{(d-1)/2}}
\sum_{\substack{x_1,\ldots,
x_d\,\in(\mathbb Z/q\mathbb Z)^\times\\x_1\cdots
x_d\equiv a\bmod q}}e\Bigl(\frac{x_1+\cdots+x_d}{q}\Bigr), \quad\text{for}\,\,(a,q)=1.
$$
Combined with the Vorono\"i formula and standard Rankin--Selberg
bounds for the dual coefficients, Deligne’s bound 
\bea\label{a5}
|\Kl_d(a;q)|\leq d^{\omega(q)},\quad\text{for}\,\,(a,q)=1
\eea
yields the so-called \emph{standard} level of distribution
$\frac{2}{d+1}-\varepsilon$ for $\lambda_{\pi}(n)$, which for $d=3$ collapses back to the $1/2$ critical line. 

Surpassing this $1/2$ threshold demands novel cancellation mechanisms 
for bilinear forms weighted by hyper-Kloosterman sums, 
a subject revitalized by foundational algebraic number theoretic work of Kowalski--Michel--Sawin \cite{KMS},
which developed a comprehensive geometric framework via 
$\ell$-adic cohomology, establishing nontrivial bounds for bilinear forms 
in hyper-Kloosterman sums valid well below the classical P\'{o}lya--Vinogradov Fourier range. 
As a key application, they analyzed the convolution coefficients
\bea\label{a1}
\lambda_{1\boxplus f}(n)=(\lambda_f \star 1)(n)=\sum_{d\mid n}\lambda_{f}(d),
\eea
where \(f\) is a level \(1\) holomorphic primitive cusp form
and \(\pi_f\) is the associated automorphic representation,
corresponding to the isobaric $\textrm{GL}_3$ representation $1\boxplus \pi_f$.
Their result achieved $\vartheta=1/2+1/102-\varepsilon$
for prime moduli.

In this paper, inspired by the works of Kowalski--Lin--Michel \cite{KLM} and Mili\'{c}evi\'{c} in~\cite{MD}, 
we obtain the following smoothed distribution result over powers of a fixed odd prime.
\begin{theorem}\label{Main theorem}
Let $f$ be a holomorphic primitive cusp form of level $1$,
trivial nebentypus and weight $k_f$,
with Hecke eigenvalues $\lambda_f(n)$,
normalized so that $|\lambda_f(n)|\leq \tau(n)$.

Let $p$ be a fixed odd prime, let $a$ be an integer coprime to $q$,
and let $X\geq 2$. For every $\eta>0$, there exists
$\delta=\delta(\eta)>0$ such that, for every $\gamma\geq2$ and every
fixed $V\in C_c^\infty(\mathbb R_{>0})$, whenever
\[
q=p^\gamma\leq X^{1/2+7/458-\eta},
\]
we have
\begin{equation}\label{main result}
\begin{split}
\Delta(X;a,q):=\mathop{\sum_{n\geq1}}_{n\equiv a\bmod q}
\lambda_{1\boxplus f}(n)V\left(\frac{n}{X}\right)
-\frac{1}{\varphi(q)} \mathop{\sum_{n\geq 1}}_{(n,q)=1}
\lambda_{1\boxplus f}(n)V\left(\frac{n}{X}\right)\ll_{f, V, \eta, p} (X/q)^{1-\delta},
\end{split}
\end{equation}
where the implied constant depends only on $f$, $V$, $\eta$ and $p$.

\end{theorem}

\begin{remark}
Under the generalized Ramanujan conjecture, the same argument applies, with minor modifications, to Hecke--Maass cusp forms.
\end{remark}

In the special case where $f$ is replaced by the non-holomorphic Eisenstein series 
whose Hecke eigenvalues are $\tau(n)$, the convolution $\lambda_f \star 1$ 
reduces to $\tau \star 1 = \tau_3$. Thus both $\tau_3$ and $\lambda_{1\boxplus f}$ 
are Hecke eigenvalues of non-cuspidal automorphic 
representations of $\mathrm{GL}_{3,\mathbb{Q}}$, that is, the isobaric 
representations $1\boxplus 1\boxplus 1$ and $1\boxplus \pi_f$, respectively. 
The techniques developed in this paper are expected to extend to a broader class of fixed non-cuspidal $\mathrm{GL}_{3,\mathbb{Q}}$ automorphic representations.
Extending this beyond the non-cuspidal setting to genuine cuspidal $\mathrm{GL}_{3,\mathbb{Q}}$ representations remains an interesting and challenging open problem.

\medskip

\noindent\textbf{Sketch of the proof.}  
The core technical challenge of our analysis reduces to controlling bilinear sums weighted by
$\textrm{Kl}_3$ and Hecke eigenvalues $\lambda_f$.
Our proof strategy partitions the dyadic ranges of summation lengths $L,M$ for the bilinear variables, 
deploying distinct analytic tools for each regime. For small \(L\), we use Sharma's estimate for sums of
\(\mathrm{GL}_2\) coefficients twisted by \(\mathrm{Kl}_3\).
For sufficiently large \(L\), Poisson summation gives the required
saving. In the remaining depth-aspect range, the explicit
prime-power evaluation of \(\mathrm{Kl}_3\), together with the
\(p\)-adic exponent-pair method of Mili\'cevi\'c, yields cancellation
beyond the direct Fourier-completion range.

%
%
%
%
\medskip
\noindent\textbf{Notation.}
Throughout the paper, the letters
$\varepsilon$ and A denote arbitrarily small and arbitrarily large positive
real numbers, respectively, not necessarily the same at each occurrence.
As usual, $e(z) = \exp (2 \pi i z) = e^{2 \pi i z}$ for any $z\in\mathbb{C}$.

If $a\in\mathbb{Z}$ and $q\geq 1$ are integers and $(a,q)=1$, we write
$\overline{a}$ for the inverse of $a$ in $(\mathbb{Z}/q\mathbb{Z})^\times$;
the modulus $q$ will always be clear from context. 
For $c\geq 1$ and $a,b$ integers, or congruence classes modulo $c$, the classical Kloosterman sum $S(a,b;c)$
is defined by
\bna
S(a,b;c)=\mathop{\sum_{x\bmod c}}_{(x,c)=1}e\left(\frac{ax+b\overline{x}}{c}\right),
\ena
where $\overline{x}$ is the inverse of $x$ modulo $c$.

The symbol $k\sim K$ means that the integer $k$ satisfies the inequalities $K< k\le 2K$.
By $f = O(g)$ for $x\in I$, or $f \ll g$ for $x\in I$, where $I$ is an arbitrary set
on which $f$ is defined, we mean synonymously that there exists a positive constant $c$ such that $|f(x)| \leq cg(x)$ for all $x\in I$. The ``implied constant" refers to any value of $c$ for which this holds. It may depend on the set $I$, which is usually specified explicitly, or clearly determined by the context.
 
\section{Preliminaries}
In this section, we present some essential information and tools needed later. Firstly, we briefly
review some basic facts of automorphic $L$-functions on $\GL_2$.
\subsection{Automorphic $L$-functions}
Let $f$ be a holomorphic primitive cusp form of level $1$,
trivial nebentypus, and weight $k_f$, with normalized Hecke
eigenvalues $\lambda_f(n)$. Let $\chi$ be a primitive Dirichlet
character modulo $q>1$. We consider the degree-$3$ $L$-function
\[
L((1\boxplus f)\times\chi,s)
=
L(\chi,s)L(f\times\chi,s),
\]
where the twisted $L$-function $L(f\times \chi, s)$ associated with $f$ is defined by
\bna
L(f\times \chi, s)=\sum_{n\geq 1}\frac{\lambda_{f}(n)\chi(n)}{n^s}
\ena
in the half plane $\Re(s)>1$.
According to \eqref{a1},
note that for $\Re(s)>1$, we have the Dirichlet series expansion
\bna
L((1\boxplus f)\times \chi, s)
=\sum_{n\geq 1}\frac{\lambda_{1\boxplus f}(n)\chi(n)}{n^s},
\ena
which has an analytic continuation to the whole complex plane $\mathbb{C}$.
Put
\bna
\kappa_\chi=
\begin{cases}
0, & \chi(-1)=1,\\
1, & \chi(-1)=-1.
\end{cases}
\ena
Define
$\Gamma_{\mathbb{R}}(s)=\pi^{-s/2}\Gamma(s/2)$ and let
$L_{\infty,\kappa_\chi}(s)=L_\infty(f,s)L_\infty(\chi,s)$, where
\bna
L_\infty(f,s)=\Gamma_{\mathbb{R}}\left(s+\frac{k_f-1}{2}\right)
\Gamma_{\mathbb{R}}\left(s+\frac{k_f+1}{2}\right),\quad
L_\infty(\chi,s)=\Gamma_{\mathbb{R}}(s+\kappa_\chi)
\ena
are the archimedean $L$-factors of $L(f,s)$ and
$L(\chi,s)$, respectively. Note that the archimedean
$L$-factor of $L(f\times\chi,s)$ coincides with that of $L(f,s)$.
Further, let 
\bna
\epsilon((1\boxplus f)\times \chi)
=i^{-\kappa_\chi}\epsilon(f)\epsilon_{\chi}^3,
\ena
where $\epsilon(f)$ is the root number of $L(f,s)$ satisfying
$|\epsilon(f)|=1$ and
\bna
\epsilon_{\chi}=q^{-1/2}\sum_{x\in (\mathbb{Z}/q\mathbb{Z})^{\times}}
\chi(x)e\left(\frac{x}{q}\right)
\ena
is the normalized Gauss sum associated with $\chi$.
Since $\chi$ is primitive, we have $|\epsilon_\chi|=1$.
Define the completed $L$-function
\bna
\Lambda((1\boxplus f)\times \chi, s)
=q^{3s/2}L_{\infty,\kappa_\chi}(s)L((1\boxplus f)\times \chi, s),
\ena
which satisfies a functional equation of the form
\bea\label{FE}
\Lambda((1\boxplus f)\times \chi, s)
=\epsilon((1\boxplus f)\times \chi)
\Lambda((1\boxplus \overline{f})\times \overline{\chi}, 1-s),
\eea
where $\overline{f}$ is the dual form of $f$ for which
$\lambda_{\overline{f}}(n)=\overline{\lambda_f(n)}$ and
$L_\infty(\overline{f},s)=\overline{L_\infty(f,\overline{s})}$,
and $\overline{\chi}$ denotes the complex conjugate of $\chi$.

For more details, we refer readers to Iwaniec--Kowalski \cite[Sections 5.1 \& 14.8]{IK}.
\subsection{Poisson summation}
We recall the Poisson summation formula over an arithmetic progression (see e.g. \cite[Eq.\,(4.24)]{IK}).

\begin{lemma}
Let $\beta\in \mathbb{Z}$ and let $q\geq 1$ be an integer.
For a Schwarz function $V:\mathbb{R}\rightarrow\mathbb{C}$, we have
\begin{equation*}\label{Poisson}
\sum_{\substack{n\in \mathbb{Z} \\ n \equiv \beta \bmod q}}V(n)
=\frac{1}{q}\sum_{n\in \mathbb{Z}}\widehat{V}\left(\frac{n}{q}\right)
e\left(\frac{n\beta}{q}\right),
\end{equation*}
where \begin{equation*}
\widehat{V}(x) = \int_{\mathbb{R}} V(u)e(-xu)\mathrm{d}u.
\end{equation*}
is the Fourier transform of $V$.
Repeated integration by parts shows 
\begin{equation*}\label{Vcheck}
\widehat{V}(x)\ll_{A,V} (1+|x|)^{-A},
\end{equation*}
for any $A\geq 0$.
\end{lemma}

\subsection{Kloosterman sums}\label{subsec:Kloosterman sum}
We first record a useful vanishing criterion for hyper-Kloosterman sums.
For later use, we extend the definition of \(\Kl_3(t;q)\)
to every integer \(t\) by setting
\[
\Kl_3(t;q)
:=\frac{1}{q}\,\,\,\sideset{}{^*}\sum_{x,y\bmod q}e\left(\frac{x+y+t\overline{xy}}{q}\right),
\]
where the superscript $*$ means that both \(x\) and \(y\)
are coprime to \(q\). Equivalently,
\[
\Kl_3(t;q)=\frac{1}{q}\,\,\,\sideset{}{^*}\sum_{x\bmod q}
e\left(\frac{t\overline{x}}{q}\right)S(1,x;q).
\]
When \((t,q)=1\), this definition agrees with the preceding
definition of the normalized hyper-Kloosterman sum.
\begin{lemma}\label{lem:Kloosterman=0-initial}
Suppose $q=p^\gamma$ with a prime $p$ and an integer  $\gamma\geqslant2$.
For any $m\in\mathbb{Z}$ with $p\mid m$ and $(a,p)=1$, we have
$$
\Kl_3(am;q)=0.
$$
\end{lemma}

\begin{proof}
Noting
\bna
\text{Kl}_3(am;p^\gamma)=\frac{1}{p^\gamma}\,\,\sideset{}{^*}
\sum_{x\bmod p^\gamma}e\left(\frac{am\overline{x}}{p^\gamma}\right)S(1,x; p^\gamma),
\ena
we can write
\bna
x=p^{\gamma-1}\alpha+\beta,\quad 1\leq\alpha\leq p,\quad 1\leq\beta\leq p^{\gamma-1},\,\, (\beta, p)=1.
\ena
and
\bna
\overline{x}\equiv\overline{\beta}-p^{\gamma-1}\overline{\beta}^2\alpha\quad(\bmod p^\gamma),
\ena
by simple calculation.
Opening the Kloosterman sum and exchanging the order of sums, we have
\bna\label{b6}
\text{Kl}_3(am;p^\gamma)&=&\frac{1}{p^\gamma}\,\,\sideset{}{^*}\sum\limits_{{\beta \bmod p^{\gamma-1}}} \sum_{\alpha \bmod p}
e\left(\frac{am(\overline{\beta}-p^{\gamma-1}\overline{\beta}^2\alpha)}{p^\gamma}\right)S(1,\beta + p^{\gamma-1} \alpha; p^\gamma)\nonumber\\
&=&\frac{1}{p^\gamma}\,\,\sideset{}{^*}\sum\limits_{{\beta \bmod p^{\gamma-1}}}
e\left(\frac{am\overline{\beta}}{p^\gamma}\right)
\sideset{}{^*}\sum_{d\bmod p^\gamma}e\left(\frac{d+\overline{d}\beta}{p^\gamma}\right)
\sum_{\alpha \bmod p}e\left(\frac{\overline{d}\alpha}{p}\right),
\ena
since $p\mid m$.
The sum over $\alpha$ above now vanishes due to the orthogonality of additive characters.
This completes the proof.
\end{proof}
Let $\mathbb{Z}_p$ denote the ring of $p$-adic integers within the field $\mathbb{Q}_p$ of $p$-adic numbers. 
The multiplicative unit group is defined by
$\mathbb{Z}_p^\times := \mathbb{Z}_p \setminus p\mathbb{Z}_p.$
We use the following evaluation of the normalized hyper-Kloosterman sums modulo prime powers,
which can be found in D\c{a}browski--Fisher~\cite[Eqs.~(1.16)--(1.17)]{DF1997}.
\begin{lemma}\label{lem:Kloosterman}
Let $q=p^\gamma$. Assume either that $p\neq3$ and
$\gamma\geq2$, or that $p=3$ and $\gamma\geq5$. For $A\in(\mathbb{Z}/p^\gamma\mathbb{Z})^\times$, define
\begin{align*}
\epsilon(A,p^\gamma)=
\begin{cases}
\Bigl(\frac{p^\gamma}{3}\Bigr), & p\neq3, \,\gamma \geq 2,\\
\sqrt3\Bigl(\frac{A}{3}\Bigr)i,& p=3, \,\gamma \geq 5.
\end{cases}
\end{align*}
Then for $a\in \mathbb{Z}_p^{\times}$, we have 
\begin{equation*}
\Kl_3(a;q)=\sum_{v \in \mathbb{Z}_p\atop v^3= a}\epsilon(v,p^\gamma)e\left(\frac{3v}{p^\gamma}\right).
\end{equation*}
\end{lemma}

\subsection{$p$-adic exponent data and pairs}
We now give a brief summary of this theory, following \cite[Sections 2\,\&\,3]{MD}.

Let the notations be as in Section~\ref{subsec:Kloosterman sum}.  Among all power series $a(t)=\sum_{k=0}^{\infty}a_kt^k$ with 
coefficients $a_k\in\mathbb{Z}_p$, we define the set
\begin{align*}
\mathbf{I}_0(\mathbb{Z}_p)=\left\{a(t):a_k\in\mathbb{Z}_p
\,(k\geqslant 0),\,\,\lim{k\rightarrow\infty} |a_k|_p=0\right\}.
\end{align*}
For any given $\lambda\in\mathbb{R}_{\geqslant 0}$, we further define the following subspaces of $\mathbf{I}_0(\mathbb{Z}_p)$:
\begin{equation*}
\mathbf{I}_0[\lambda](\mathbb{Z}_p)=\{a(t)\in\mathbf{I}_0(\mathbb{Z}_p):\text{ord}_pa_k\geqslant\lceil k\lambda\rceil\,\,(k\in\mathbb{N}_0)\}.
\end{equation*}
We set $\rho_p = 1/(p-1)$. For $y\in\mathbb{Q}^{+}$, 
we define $\rho_p(y) = \rho_p$ if $\ord_p(y) \neq 0$ and $0$ otherwise.
Moreover, for $y\in\mathbb{Q}^{+}$, let $\iota(y)=\max(0, \text{ord}_p(y^{-1}))$ and $\iota'(y)=\max(0, \text{ord}_p\,y)$, so that $\text{ord}_p\,y=\iota'(y)-\iota(y)$.
\begin{definition}\label{Definition}
Let $w \in \mathbb{Z}$, $u, \kappa \in \mathbb{N}$ with $\kappa \geq 1 + \iota'(2)$, $\lambda \in \rho_p \mathbb{N}$, and $y \in \mathbb{Q}^+$. Let $\iota = \iota(y)$, $\iota' = \iota'(y)$, and let $\omega, \omega' \in \mathbb{Z}_p^\times$. A power series $f \in \mathbb{Q}_p^\times \mathbf{I}_0(\mathbb{Z}_p)$ is said to belong to the class $\mathbf{F}(w, y, \kappa, \lambda, u, \omega, \omega')$ if its derivative satisfies
\begin{equation}
\label{ClassFDefinition}
f'(t) = p^w \omega' \big( 1 + p^{\,\iota + \kappa} \omega t \big)^{-y} + p^w \gamma_0 + p^{\,u + w} g(t)
\end{equation}
for some $\gamma_0 \in \mathbb{Z}_p$ and some $g \in \mathbf{I}_0[\lambda](\mathbb{Z}_p)$.
We say that $f$ belongs to the class $\mathbf{F}(w, y, \kappa, \lambda, u)$ if $f \in \mathbf{F}(w, y, \kappa, \lambda, u, \omega, \omega')$ for some $\omega, \omega' \in \mathbb{Z}_p^\times$.
\end{definition}

We now define $p$-adic exponent data and pairs. Let $P$ denote the set of prime numbers. For any sets $X$, $Y$, and any family of subsets $X_p \subset X$ ($p \in P$), let $\mathbf{J}(X_p;Y)$ be the set of all functions $g : \mathbb{Q}^{+} \times \bigsqcup_{p \in P} ({p} \times X_p) \to Y$ such that, for every $y \in \mathbb{Q}^{+}$, there exists a finite subset $P_0(y) \subset P$ and a function $g_0 : (P \setminus P_0(y)) \times X \to Y$ such that $g(y,p,x) = g_0(p,x)$ for every $p \in P \setminus P_0(y)$ and every $x \in X_p$.

In particular, write $\mathbf{J}(Y) := \mathbf{J}(\emptyset;Y)$ for the set of functions 
$g(y,p) : \mathbb{Q}^{+} \times P \to Y$ with the above properties, and 
$\mathbf{J}_1(Y) := \mathbf{J}(\mathbb{N}_p' \times \rho_p \mathbb{N};Y)$ 
(with $X = \mathbb{R}^{+}$ and $\mathbb{N}'_p = \iota'(2) + \mathbb{N}$) for the set of such functions $g(y,p,\kappa,\lambda) : \mathbb{Q}^{+} \times \bigsqcup_{p \in P} ({p} \times \mathbb{N}'_p \times \rho_p \mathbb{N}) \to Y$.

\begin{definition}\label{exponent pair}
Let $k,\ell\in\mathbb{R}$ with $0\leqslant k\leqslant\frac12\leqslant\ell\leqslant 1$, $r\in\mathbf{J}_1(\mathbb{R})$, $\delta\in\mathbb{R}^{+}_0$, $n_0,u_0\in\mathbf{J}_1(\mathbb{N})$, $\kappa_0\in\mathbf{J}(\mathbb{N})$, $\lambda_0\in\mathbf{J}(\mathbb{R}^{+}_0)$, and $n_0(y,p,\kappa,\lambda)>\kappa+\iota'(y)$.
We say that a quintuple $$\big(k,\,\ell,\,\,r,\,\delta,\,\,(n_0,u_0,\kappa_0,\lambda_0)\big)$$ is a $p$-adic exponent datum if, for every $p\in P$, $y\in\mathbb{Q}^{+}$, $w\in\mathbb{Z}$, $\kappa\in\mathbb{N}$ with $\kappa\geqslant 1+\iota'(2)$, $\lambda\in\rho_p\mathbb{N}$, $n,u\in\mathbb{N}$ such that
$$ \kappa\geqslant\kappa_0(y,p),\quad\lambda\geqslant\lambda_0(y,p),\quad n\geqslant w+n_0(y,p,\kappa,\lambda),\quad u\geqslant u_0(y,p,\kappa,\lambda), $$
and for every $f\in\mathbf{F}(w,y,\kappa,\lambda,u)$, $M\in\mathbb{Z}$, and $0<H\leqslant p^{n-w-\kappa-\iota'}$, we have the estimate
\begin{equation}
\label{SharpDef}
\sum_{M < m\leqslant M+H}e\left(\frac{f(m)}{p^n}\right)\ll p^{r}\left(\frac{p^{n-w-\kappa-\iota'}}H\right)^kH^{\ell}
\big(\log p^{n-w-\kappa-\iota'}\big)^{\delta},
\end{equation}
where $r=r(y,p,\kappa,\lambda)$, and the implied constant depends only on the datum.
We say that a pair
$$ (k,\ell) $$
of non-negative numbers is a $p$-adic exponent pair if $\big(k,\ell,r,\delta,(n_0,u_0,\kappa_0,\lambda_0)\big)$ is a $p$-adic exponent datum for some $r\in\mathbf{J}_1(\mathbb{R})$, $\delta\in\mathbb{R}^{+}_0$, $n_0,u_0\in\mathbf{J}_1(\mathbb{N})$, $\kappa_0\in\mathbf{J}(\mathbb{N})$, $\lambda_0\in\mathbf{J}(\mathbb{R}^{+}_0)$.
\end{definition}
Note that $(0,1)$ is trivially a $p$-adic exponent pair, as $(0,1,0,0,(\kappa+\iota'+1,1,1+\iota'(2),\rho_p))$ is a $p$-adic exponent datum. 

We show how to generate such $p$-adic exponent data by iterating the $A$- and $B$-processes \cite[Theorems 5\,\&\,4]{MD}. More precisely, if $\big(k,\ell,r,\delta,(n_0,u_0,\kappa_0,\lambda_0)\big)$ is a $p$-adic exponent datum, then the following are also $p$-adic exponent data:
\begin{align*}
&A\big(k,\ell,r,\delta,(n_0,u_0,\kappa_0,\lambda_0)\big) = \left( \frac{k}{2(k+1)}, \frac{k+\ell+1}{2(k+1)}, \tilde{r}, \tilde{\delta}, (\tilde{n}_0, \tilde{u}_0, \tilde{\kappa}_0, \tilde{\lambda}_0) \right),\\
&B\big(k,\ell,r,\delta,(n_0,u_0,\kappa_0,\lambda_0)\big) = \left( \ell-\frac{1}{2}, k+\frac{1}{2}, \tilde{r}, \tilde{\delta}, (\tilde{n}_0, \tilde{u}_0, \tilde{\kappa}_0, \tilde{\lambda}_0) \right),
\end{align*}
where the parameters $\tilde{r}$, $\tilde{\delta}$ and $(\tilde{n}_0, \tilde{u}_0, \tilde{\kappa}_0, \tilde{\lambda}_0)$ are as defined in \cite[Theorems 5\,\&\,4]{MD}.
In practice, these parameters are less critical, and it should be noted that their specific values differ between the $A$- and $B$-processes.
Applying the $A$- and $B$-processes to the initial datum $\omega_{01}$, 
we obtain, after some computation, the new datum
\begin{align*}
A(\omega_{1/2})=AB(\omega_{01})&=\Big(\tfrac16,\tfrac23,\tfrac16(1-\kappa),\tfrac12,\\
&\qquad\qquad\big(\big\lceil\max(5\kappa-1,\varepsilon_u(\tfrac{13}3+\tfrac53\kappa-3\lfloor\tilde{\lambda}\rfloor)) \big\rceil,\\
&\qquad\qquad\hphantom{\big(}\max(2\kappa+\lceil\tilde{\lambda}\rceil-\lfloor\lambda\rfloor
-2\lfloor\tilde{\lambda}\rfloor+1,1),1,2\rho_p\big)\Big),
\end{align*}
which is valid for all primes $p\not\in\{2,3\}$ such that $\operatorname{ord}_p(y) = \operatorname{ord}_p(y+1) = 0$. Here, $\tilde{\lambda} = \min( \kappa - \rho_p(y), \lambda )$, and $\varepsilon_u$ is defined as $\varepsilon_u=0$ if $u_0(y^{\pm}+1,p,\kappa,\tilde{\lambda})-\kappa+\iota'(2)+\varepsilon(y^{\pm})\leqslant 0$, and $\varepsilon_u=1$ otherwise.

The $p$-adic exponent data simplify under the restriction $\kappa = \tilde{\lambda}$, which is equivalent to $\lambda \geqslant \kappa$ and $\rho_p(y) = 0$. With this simplification, we have the following exponent data:
\begin{align*}
\omega_{01}[\kappa=\tilde{\lambda}]&=\big(0,1,0,0,(\kappa+1,1,1,\rho_p)\big),\\
&\qquad \ord_py=0,\\
B(\omega_{01})[\kappa=\tilde{\lambda}]&=\big(\tfrac12,\tfrac12,0,1,(\kappa+1,1,1,\rho_p)\big),\\
&\qquad p\not\in\{2,3\},\,\,\ord_py=0,\\
AB(\omega_{01})[\kappa=\tilde{\lambda}]&=\big(\tfrac16,\tfrac23,\tfrac16(1-\kappa),\tfrac12,(5\kappa-1,1,1,2\rho_p)\big),\\
&\qquad p\not\in\{2,3\},\,\,\ord_p\{y,y+1\}=0,\\
 BA^3B(\omega_{01})[\kappa=\tilde{\lambda}]&=\big(\tfrac{11}{30},
 \tfrac{8}{15},\tfrac1{10}(1-\kappa),\tfrac12,
 (\lceil\tfrac{23}2\kappa-6\rceil,1,1,2\rho_p)\big),\\
&\qquad p\not\in\{2,3\},\,\,{\rm ord}_p\{y,y+1,y+2,y+3,\\
& \qquad 2y+1,2y+3,3y+1,3y+2\}=0,\\
ABA^3B(\omega_{01})[\kappa=\tilde{\lambda}]&=\big(\tfrac{11}{82},
 \tfrac{57}{82},\tfrac7{41}(1-\kappa),\tfrac12,
 (\lceil\tfrac{71}4\kappa-\frac{39}{4}\rceil,1,1,2\rho_p)\big),\\
&\qquad p\not\in\{2,3\},\,\,{\rm ord}_p\{y,y+1,y+2,y+3,y+4,\\
&\qquad 2y+1,2y+3,2y+5,3y+1,3y+2,3y+3,\\&\qquad3y+4,3y+5,4y+1,4y+3,5y+2,5y+3\}=0.
\end{align*}
With a supply of $p$-adic exponent data, we can go beyond the limitations of direct Fourier-analytic methods.

\section{Proof of Theorem \ref{Main theorem}}
First, we make several preliminary reductions in the proof of Theorem~\ref{Main theorem}.
Put
\[
\mathcal{S}_\chi(X):=\sum_{n\geq 1}
\lambda_{1\boxplus f}(n)\chi(n)V\left(\frac{n}{X}\right).
\]
Since $(a,q)=1$, the orthogonality relation for Dirichlet
characters $\chi\bmod q$ gives that the first term of $\Delta(X;a,q)$ is equal to
\begin{equation*}
\begin{split}
\frac{1}{\varphi(q)}\sum_{\chi\bmod q}\overline{\chi}(a)\mathcal{S}_\chi(X)
=\frac{1}{\varphi(q)} \mathop{\sum_{n\geq 1}}_{(n,q)=1}
\lambda_{1\boxplus f}(n)V\big(\frac{n}{X}\big)+\frac{1}{\varphi(q)}
\mathop{\sum_{\chi\bmod  q}}_{\chi\neq\chi_0}
\overline{\chi}(a)\mathcal{S}_{\chi}(X).
\end{split}\end{equation*}
Consequently,
\begin{equation}\label{character-decomposition}
\Delta(X;a,q)
=
\frac{1}{\varphi(q)}
\sum_{\substack{\chi\bmod q\\ \chi\neq\chi_0}}
\overline{\chi}(a)\mathcal{S}_\chi(X).
\end{equation}

Next, let $q=p^\gamma$ with $\gamma\geq2$. The non-principal
imprimitive characters modulo $q$ are precisely the lifts of the
non-principal characters modulo $q/p$. Therefore, their total
contribution to \eqref{character-decomposition} is
\[
\frac{\varphi(q/p)}{\varphi(q)}
\Delta\left(X;a,\frac{q}{p}\right)
=
\frac{1}{p}\Delta\left(X;a,\frac{q}{p}\right).
\]
For convenience, we define
\[
\Delta^{\ast}(X;a,q):=\frac{1}{\varphi(q)}\,\,
\sideset{}{^\ast}\sum_{\chi\bmod q}
\overline{\chi}(a)\mathcal{S}_\chi(X),
\]
where the superscript $\ast$ indicates that the summation is
restricted to primitive characters.  We have therefore obtained
\begin{equation}\label{primitive-decomposition}
\Delta(X;a,q)=\Delta^{\ast}(X;a,q)+\frac{1}{p}\Delta\left(X;a,\frac{q}{p}\right),
\qquad \gamma\geq 2.
\end{equation}
The second term in \eqref{primitive-decomposition} can be handled
by induction on $\gamma$ (The base case is established in Remark \ref{remark 2} below). Indeed, assuming the desired estimate \eqref{main result} for
$q/p$, we obtain
\[
\begin{split}
\frac{1}{p}\Delta\left(X;a,\frac{q}{p}\right)
&\ll_{f,\eta,A}\frac{1}{p}\,\left(\frac{X}{q/p}\right)^{1-\delta}
=p^{-\delta}\left(\frac{X}{q}\right)^{1-\delta}
\ll\left(\frac{X}{q}\right)^{1-\delta}.
\end{split}
\]
Thus we are led to estimate $\Delta^{\ast}(X;a,q)$.
By Mellin inversion
\bna
V(x)=\frac{1}{2 i\pi}\int_{(1+\frac{1}{10})}\widetilde{V}(s)x^{-s}\mathrm{d}s,
\ena
where $\widetilde V(s)=\int_0^\infty V(x)x^{s-1}\mathrm{d}x$
denotes the Mellin transform of $V$, we get that 
\begin{equation*}
\begin{split}
\Delta^{\ast}(X;a,q)=\frac{1}{\varphi(q)}\,\,
\sideset{}{^\ast}\sum_{\chi\bmod  q}\overline{\chi}(a)
\frac{1}{2 i\pi}\int_{(1+\frac{1}{10})}\widetilde{V}(s)\frac{\Lambda((1\boxplus f)\times\chi,s)}
{L_\infty(s)}\left(\frac{X}{q^{3/2}}\right)^s\mathrm{d}s.
\end{split}\end{equation*}
Applying the functional equation \eqref{FE}, and making the change of variable
$s\mapsto 1-s$, we get that $\Delta^{\ast}(X;a,q)$ equals 
\begin{equation*}
\begin{split}
&\frac{q^{3/2}}{\varphi(q)}\epsilon(f)\sideset{}{^\ast}\sum_{\chi\bmod  q}
\overline{\chi}(a) i^{-\kappa_\chi}\epsilon_\chi^3\,\frac{1}{2 i\pi}\int_{(1+\frac{1}{10})}\widetilde{V}(s)
L((1\boxplus \overline{f})\times \overline{\chi},1-s)
\frac{L_{\infty,\kappa_\chi}(1-s)}
{L_{\infty,\kappa_\chi}(s)}\left(\frac{X}{q^3}\right)^s\mathrm{d}s\\
=&\frac{X}{q^{3/2}\varphi(q)}\epsilon(f)
\sideset{}{^\ast}\sum_{\chi\bmod  q}
\overline{\chi}(a) i^{-\kappa_\chi} \epsilon_\chi^3\,\frac{1}{2 i\pi}\int_{(-\frac{1}{10})}
L((1\boxplus \overline{f})\times \overline{\chi},s)
\frac{L_{\infty,\kappa_\chi}(s)}{L_{\infty,\kappa_\chi}(1-s)}\widetilde
V(1-s)\left(\frac{X}{q^3}\right)^{-s}\mathrm{d}s
\end{split}\end{equation*}
Shifting the contour back to $\Re(s)=\frac{3}{2}$
without going through any poles we obtain
\begin{equation}\label{primitive-dual-expression}
\begin{split}
\Delta^{\ast}(X;a,q)=\frac{X}{q^{3/2}\varphi(q)}\epsilon(f)
\sum_{n\geq1}\lambda_{1\boxplus\overline f}(n)
\sideset{}{^\ast}\sum_{\chi\bmod q}
\overline{\chi}(an)i^{-\kappa_\chi}\epsilon_\chi^3
\widecheck V_{\kappa_\chi}
\left(\frac{n}{q^3/X}\right),
\end{split}
\end{equation}
where, for $\kappa\in\{0,1\}$,
\begin{equation*}\label{dual-weight-functions}
\widecheck V_\kappa(y)
:=\frac{1}{2\pi i}\int_{(\frac{3}{2})}\frac{L_{\infty,\kappa}(s)}{L_{\infty,\kappa}(1-s)}
\widetilde V(1-s)y^{-s}\,\mathrm{d}s.
\end{equation*}
Both $\widecheck V_0$ and $\widecheck V_1$ decay rapidly as
$y\to\infty$.
\begin{remark}\label{remark 2}
We now verify the base case needed for the induction in
\eqref{primitive-decomposition}. If $p=2$, then
$\Delta(X;a,p)=0$. Suppose that $p$ is odd. Since $\gamma\geq2$
and
\[
q=p^\gamma\leq X^{1/2+7/458-\eta},
\]
we have
\[
p\leq q^{1/2}\leq X^{1/4+7/916-\eta/2},
\qquad
\frac{p^3}{X}\leq X^{-52/229-3\eta/2}<1.
\]
Every non-principal character modulo $p$ is primitive. Applying
Mellin inversion and the functional equation to each
$\mathcal S_\chi(X)$, as in the derivation of
\eqref{primitive-dual-expression}, and using the rapid decay of the
dual weights, we obtain, for every $B>0$,
\[
\begin{split}
\Delta(X;a,p)
&\ll_{f,V,B}\frac{X}{p^{3/2}}\sum_{n\geq1}
|\lambda_{1\boxplus\overline f}(n)|
\left(1+\frac{nX}{p^3}\right)^{-B-2}\ll_{f,V,B}
\frac{X}{p^{3/2}}\left(\frac{p^3}{X}\right)^B,
\end{split}
\]
where we used
$|\lambda_{1\boxplus\overline f}(n)|\leq \tau_3(n)$.
Consequently, by choosing $B$ sufficiently large, we obtain
\[
\Delta(X;a,p)\ll_{f,V,\eta,A}X^{-A}
\]
for every $A>0$. This proves the required base case.
\end{remark}
For $(t,q)=1$, set
\[
\mathcal T_q(t)
:=\sideset{}{^\ast}\sum_{\chi\bmod q}
\overline{\chi}(t)\epsilon_\chi^3.
\]
Since $q=p^\gamma$ with $\gamma\geq2$, the Gauss sum of every
imprimitive character modulo $q$ vanishes. We may therefore extend
the above summation to all characters modulo $q$. 
The orthogonality of characters then gives
\begin{equation}\label{Gauss-Kloosterman-identity}
\begin{split}
\mathcal T_q(t)&=q^{-3/2}\sum_{x,y,z\in(\mathbb Z/q\mathbb Z)^\times}
e\left(\frac{x+y+z}{q}\right)\sum_{\chi\bmod q}\chi(xyz)\overline{\chi}(t)\\
&=\varphi(q)q^{-3/2}\sum_{\substack{x,y,z\in(\mathbb Z/q\mathbb Z)^\times\\xyz\equiv t\bmod q}}e\left(\frac{x+y+z}{q}\right)
=\varphi(q)q^{-1/2}\Kl_3(t;q).
\end{split}
\end{equation}
If $(t,q)>1$, both sides of
\eqref{Gauss-Kloosterman-identity} vanish. Thus
\eqref{Gauss-Kloosterman-identity} is valid for every integer $t$.
Separating the even and odd characters in
\eqref{primitive-dual-expression}, we obtain
\[
\begin{split}
\sideset{}{^\ast}\sum_{\chi\bmod q}
\overline{\chi}(t)i^{-\kappa_\chi}\epsilon_\chi^3
\widecheck V_{\kappa_\chi}(y)=
\frac{\mathcal T_q(t)+\mathcal T_q(-t)}{2}\widecheck V_0(y)
-\frac{i\bigl(\mathcal T_q(t)-\mathcal T_q(-t)\bigr)}{2}\widecheck V_1(y).
\end{split}
\]
It follows from \eqref{Gauss-Kloosterman-identity} that
\begin{equation*}\label{primitive-Voronoi}
\begin{split}
\Delta^{\ast}(X;a,q)
=&\frac{X}{2q^2}\epsilon(f)\sum_{n\geq1}\lambda_{1\boxplus\overline f}(n)
\Bigg(\big(\Kl_3(an;q)+\Kl_3(-an;q)\big)\widecheck V_0\left(\frac{n}{q^3/X}\right)
\\&\hspace{36mm}
-i\big(\Kl_3(an;q)-\Kl_3(-an;q)\big)\widecheck V_1\left(\frac{n}{q^3/X}\right)\Bigg).
\end{split}
\end{equation*}

In conclusion, to prove Theorem \ref{Main theorem}, it suffices to show that for some $\delta>0$,
\begin{equation*}
    \label{MNsums1}
    \frac{X}{q^{2}}\epsilon(f)
\sum_{n\geq 1}\lambda_{1\boxplus \overline{f}}(n)
\mathrm{Kl}_3(an;q)
\widecheck{V}\left(\frac{n}{q^3/X}\right)\ll_{f,V,\eta, p}(X/q)^{1-\delta},
\end{equation*}
where $\widecheck V$ is a rapidly decreasing function.
Combined with \eqref{a1}, we have
\begin{equation*}
\begin{split}
\frac{X}{q^{2}}\epsilon(f)
\sum_{n\geq 1}\lambda_{1\boxplus \overline{f}}(n)
\mathrm{Kl}_3(an;q)
\widecheck{V}\left(\frac{n}{q^3/X}\right)=
\frac{X}{q^2}\epsilon(f)\sum_{l,m\geq 1}\lambda_{\overline{f}}(m)
\mathrm{Kl}_3(al m;q)\widecheck{V}\left(\frac{l m}{q^3/X}\right).
\end{split}
\end{equation*}
Let $V_i(y)\in \mathcal{C}_c^\infty(1/2, 5/2), i=1,2$, be a smooth supported function satisfying $V_i(y)=1$ for
$y\in [1,2]$ and $V_i^{(j)}(y)\ll_j 1$ for any integer $j\geq 0$.
By inserting a dyadic partition of unity for the $l$-sum and $m$-sum, respectively, we are reduced to considering sums of the form
\begin{equation}
    \label{MNsums}
   \mathcal{B}(L,M) := \frac{X}{q^2}\epsilon(f)\sum_{l\geq 1}\sum_{m\geq 1}
    \lambda_{\overline{f}}(m)\mathrm{Kl}_3(al m;q)V_1\big(\frac{l}{L}\big)V_2\big(\frac{m}{M}\big)
\end{equation}
for $O(\log^2 X)$ many real numbers $L,M\geq 1$ satisfying
\begin{equation}
    \label{LMupperbound}
    LM\ll \frac{q^3}{X}.
\end{equation}
According to \eqref{a5} and the Rankin--Selberg bound, we obtain that \eqref{MNsums} can be trivially bounded by
\begin{equation*} \label{MNboundtrivial}
    \frac{X^{1+o(1)}}q\big(\frac{LM}{q}\big),
\end{equation*}
which is good enough if $q\leq X^{1/2-\eta}$,
and that henceforth one assumes that $q\geq X^{1/2-\eta}$.
From this, we may assume that
\begin{equation*}\label{LMlower}
    LM\geq q^{1-\eta}
\end{equation*}
for some fixed $\eta>0$ that can be chosen as small as necessary. In particular
$$L+M\geq q^{1/2-\eta}.$$
To obtain nontrivial cancellation for the sum
\eqref{MNsums}, we split the argument into several cases.

\subsection{The case of small $L$}
Recall $LM\ll q^3/X$ in \eqref{LMupperbound} and $L\geq 1$.
If $L$ is small, we directly apply the following result.
In \cite[Theorem 1.3]{Sharma}, setting $\mathcal{N}=\{1\}$ and $\alpha_1=1$, we obtain
\begin{lemma}\label{sharma}
Let $q=p^\gamma$ with $p$ an odd prime and $\gamma \geq 2$. Assume $(l,q)=1$. We have
\bna
\sum_{m \geq 1}\lambda_f(m)\Kl_3(ml; q) V_2(\frac{m}{M}) 
\ll p^{7/12}M^{1/2}q^{1/3+\varepsilon}\bigg(1+\frac{M}{q}\bigg)^{2/3}+q^{13/20+\varepsilon}.
\ena
\end{lemma}
By Lemmas \ref{lem:Kloosterman=0-initial} and \ref{sharma}, we find that
\begin{align}\label{upbound1}
\mathcal{B}(L, M) & \,\ll_{\overline{f}} \frac{X^{1+o(1)}}{q^2}L
\Bigl(p^{7/12}M^{1/2}q^{1/3}(1+M/q)^{2/3}+q^{13/20}\Bigr) \notag\\
& \, \ll \frac{X^{1+o(1)}}{q}\Bigl(p^{7/12}\frac{L^{1/2}q^{5/6}}{X^{1/2}}+p^{7/12}\frac{q^{13/6}}{X^{7/6}L^{1/6}}
+\frac{L}{q^{7/20}}\Bigr).
\end{align}
In particular, this bound is suitable as long as
$$q\leq \min\{X^{3/5-\eta}L^{-3/5}, X^{7/13-\eta}\}$$
for some fixed $\eta>0$.

\subsection{The case of large $L$}
In this subsection, we fix a parameter $\kappa\in\mathbb{N}$ 
satisfying the condition in Lemma~\ref{Kloosterman sums bound},
and further distinguish two subcases according to whether $q$ is comparable  
in size to $Lp^{\kappa}$; this distinction is motivated by \eqref{condition} below.
\subsubsection{The case of $q \leq Lp^{\kappa}$}
In this case we use the Poisson summation formula in \eqref{MNsums} for $l$-sum and obtain
\begin{align}
\mathcal{B}(L, M) 
&\,=\frac{X}{q^2}\epsilon(f)\sum_{m\geq 1}\lambda_{\overline{f}}(m)
\left(\frac{L}{q}\sum_{\tilde{l}\in \mathbb{Z}}
\delta_{(\tilde l,q)=1}S(1,-am\overline{\tilde{l}};q)
\widehat{V_1}(\frac{\tilde{l}}{q/L})\right)V_2(\frac{m}{M})
\nonumber\\\noalign{\vskip 2,8mm}
&\,\ll_{\overline{f}}
\frac{X^{1+o(1)}}{q}\frac{M}{q^{1/2}}\ll_{\overline{f}} \frac{X^{1+o(1)}}{q}\frac{q^{5/2}}{XL},\label{MNbound2}
\end{align}
 where the estimate draws on Weil's bound of Kloosterman sum.

\subsubsection{The case of $q \geq Lp^{\kappa}$}
In this case, we improve the trivial bound by applying a lemma that employs a $p$-adic exponent datum.
\begin{lemma}\label{Kloosterman sums bound}
Let $q=p^\gamma$, where $p$ is an odd prime, 
and assume $(m,q)=1$. For $m\geq1$ and $V\in C_c^\infty(\mathbb{R}_{>0})$, we define 
$$
S(m,L):=\sum_{l\geq 1}\Kl_3(ml;q)V\Bigl(\frac{l}{L}\Bigr).
$$
Suppose that $ \bigl(k,\ell,r,\delta,(n_0,u_0,\kappa_0,\lambda_0)\bigr) $ 
is a $ p $-adic exponent datum as in Definition~\ref{exponent pair}. 
Then, for every $\kappa\geq
\max\big(\kappa_0(2/3,p), 1+\iota'(2)\big),$
with $r=r(2/3,p,\kappa,\infty)$, assume that
\begin{equation*}
\gamma\geq\max\big(n_0(2/3,p,\kappa,\infty)+\kappa,2\kappa\big),\qquad
p^\kappa\leq L\leq qp^{-\kappa},
\end{equation*}
when $p\geq5$.
When $p=3$, assume that
$$
\gamma\geq\max\big(5,n_0(2/3,3,\kappa,\infty)+\kappa+1,2\kappa+1\big),\qquad
3^{\kappa+1}\leq L\leq q3^{-\kappa}.
$$
Then
$$
S(m,L)\ll_Vp^{r+(1-k-\ell)\kappa}q^kL^{\ell-k}(\log q)^\delta.
$$
\end{lemma}

\begin{proof}
We closely follow the proof presented in~\cite[Theorem 6]{MD}.
First, we treat the hyper-Kloosterman sum $\Kl_3(ml;q)$. According to Lemma~\ref{lem:Kloosterman},
for prime power $q = p^\gamma$ and $p \nmid ml$, $\Kl_3(ml;q)=0$ unless $ml\in \mathbb{Z}_p^{\times 3}$, in which case we have the explicit formula
\[
\Kl_3(ml;q) = \sum_{\substack{v \in \mathbb{Z}_p \\ v^3 = ml}} \epsilon(v,q) e\left(\frac{3v}{q}\right).
\]
Suppose now that $v^3=ml \in \mathbb{Z}_p^{\times}$. Then the number of solutions of $v^3=ml$ in $\mathbb{Z}_p^\times$ depends on $p\bmod3$:
\begin{itemize}
\item[-]
exactly 3 solutions if $p\equiv 1 \bmod 3$;
\vskip 2mm
\item[-]
exactly 1 solution if $p\equiv 2 \bmod 3$.
\end{itemize}
Let
$$
d_p=
\begin{cases}
3, & p\equiv1\bmod3,\\
1, & p\equiv2\bmod3\quad \text{or}\,\, p=3.
\end{cases}
$$
If \(p\equiv1\pmod3\), let \(u_0\in\mathbb Z_p^\times\) be a fixed
primitive cube root of unity; otherwise, put \(u_0=1\).
For every \(x\in\mathbb Z_p^{\times3}\), fix a cube root
\(s(x)\in\mathbb Z_p^\times\) satisfying $s(x)^3=x.$
When \(p\neq3\), Hensel's lemma implies that the solutions of
\(v^3=x\) in \(\mathbb Z_p^\times\) are precisely
$$
v=s(x)u_0^j,\qquad 0\le j\le d_p-1.
$$
When \(p=3\), the cubing map sends \(1+3\mathbb Z_3\) bijectively
onto \(1+9\mathbb Z_3\). Hence every
\(x\in\mathbb Z_3^{\times3}\) has a unique cube root in
\(\mathbb Z_3^\times\), so the same notation is valid with \(d_3=1\).
Consequently, whenever \(ml\in\mathbb Z_p^{\times3}\),
$$
\Kl_3(ml;p^\gamma)=\sum_{0\le j\le d_p-1}
\epsilon\left(s(ml)u_0^j,p^\gamma\right)
e\left(\frac{3s(ml)u_0^j}{p^\gamma}\right),
$$
where in the case \(p=3\) the sum consists of a single term.

Next, for $p\geq 5$, we proceed by splitting the sum into residue classes modulo $p^{\kappa}$ for a suitable $\kappa\leq\gamma$. 
For each $c$ with $1\leq c\leq p^\kappa$ and $p\nmid c$, let
$c'\in\mathbb Z_p^\times$ denote the $p$-adic inverse of $c$, so that
$cc'=1.$
Since \(p\neq3\), the cubing map is an automorphism of
\(1+p^\kappa\mathbb Z_p\). Hence
$m(c+p^\kappa t)\in\mathbb Z_p^{\times3}\Longleftrightarrow mc\in\mathbb Z_p^{\times3},$
and, in this case, its cubic roots are
$$
s(mc)u_0^j(1+p^\kappa c't)^{1/3}, \qquad 0\leq j\leq d_p-1.
$$
This yields the decomposition
\begin{align}\label{SumAspAdic}
S(m,L)&\,=\sum_{0 \leq j \leq d_p - 1}\sum_{\substack{1 \leq c \leq p^\kappa\\ p \nmid c\\ mc\in\mathbb Z_p^{\times3}}} \epsilon(s(mc)u_0^j,p^\gamma)
\sum_{l=0}^\infty e\left(\frac{f_{c, j}(l)}{p^\gamma}\right) V\left(\frac{c + p^\kappa l}{L}\right) \notag\\
&\,= \sum_{0 \leq j \leq d_p - 1}\sum_{\substack{1 \leq c \leq p^\kappa\\ p \nmid c\\ mc\in\mathbb Z_p^{\times3}}}\epsilon(s(mc)u_0^j,p^\gamma)
\sum_{l=1}^\infty e\left(\frac{f_{c, j}(l)}{p^\gamma}\right) V\left(\frac{c + p^\kappa l}{L}\right) + O(p^{\kappa}),
\end{align}
where the phase is given by
\begin{equation*}
f_{c, j}(l) = 3 s(mc) u_0^{j} (1 + p^{\kappa} c' l)^{1/3}.
\end{equation*}
Note that for any $\omega\in\mathbb{Z}^{\times}_p$, the function $(1+p^{\kappa}\omega l)^{1/3}$ lies in $1+p^{\kappa}\mathbf{I}_0(\mathbb{Z}_p)$. 
Moreover, its derivative is given by $\left[(1+p^{\kappa}\omega l)^{1/3}\right]'=\frac{1}{3}p^{\kappa}\omega\big(1+p^{\kappa}\omega l\big)^{-2/3}$. 
Therefore, $(1+p^{\kappa}\omega l)^{1/3}$ belongs to the class $\mathbf{F}(\kappa,2/3,\kappa,\infty,\infty,\omega,\omega/3)$.
Here, the symbol $\infty$ is used in place of $\lambda$ and $u$ to indicate that $f$ satisfies Definition~\ref{Definition} for arbitrarily large values of $\lambda$ and $u$. This means that condition \eqref{ClassFDefinition} holds with $g=0$. 
In particular, we have that the phase $f_{c, j}(l) = 3 s(mc) u_0^{j} (1 + p^{\kappa} c' l)^{1/3}$ satisfies
\begin{equation*}
f_{c, j}(l) \in \mathbf{F}(\kappa, 2/3, \kappa,\infty, \infty, c', s(mc) u_0^{j}c').
\end{equation*}

We estimate the innermost sum in \eqref{SumAspAdic} using summation by parts and denote it by $S(c,j)$.  
Define  
\begin{equation*}
\tilde{S}(t) := \sum_{1 \le l \le t} e\left(\frac{f_{c,j}(l)}{q}\right),
\end{equation*}
if $ t \geq 1 $, and $ \tilde{S}(t) = 0 $ for $ t < 1 $. 
Since
$$ \gamma-\kappa\geqslant n_0,\quad \kappa\geqslant\kappa_0, $$
we apply the $p$-adic exponential datum from~\eqref{SharpDef} to bound $\tilde{S}(t)$ and obtain
\begin{equation}\label{St bound}
\tilde{S}(t)\ll p^r\left(\frac{p^{\gamma-2\kappa}}{t}\right)^k t^{\ell}(\log q)^{\delta}
\end{equation}
for \(1 \leq t \leq p^{\gamma-2\kappa}\), and, more generally, for all \(t \geq 0\),
\[
\tilde{S}(t) \ll p^r \left( p^{(\gamma-2\kappa)k} t^{\ell-k} + \frac{t}{p^{(\gamma-2\kappa)(1-\ell)}} \right) (\log q)^\delta.
\]
Using summation by parts, we obtain
\[
\begin{aligned}
S(c,j) = & \, \int_{1-0}^\infty V\left(\frac{c + p^\kappa t}{L}\right) \dd \tilde{S}(t) = V\left(\frac{c + p^\kappa t}{L}\right) \tilde{S}(t) 
\bigg|_{1-0}^\infty - \int_{1}^\infty \tilde{S}(t) \dd V\left(\frac{c + p^\kappa t}{L}\right) \\
\ll & \, p^r (\log q)^\delta \int_1^\infty \left( p^{(\gamma -2\kappa)k} t^{\ell-k} + \frac{t}{p^{(\gamma-2\kappa)(1-\ell)}} \right) \frac{p^{\kappa}}{L}
\left| V' \left( \frac{c+p^\kappa t}{L} \right) \right| \dd t.
\end{aligned}
\]
Introducing a substitution $t=(L\xi-c)p^{-\kappa}$, we find that
\[
S(c,j) \ll p^{r} (\log q)^{\delta} \int_{\xi_0}^{\infty} \left( p^{(\gamma-2\kappa)k}(Lp^{-\kappa})^{\ell-k} 
\xi^{\ell-k} + p^{(\gamma-2\kappa)(\ell-1)}(Lp^{-\kappa}) \xi \right)  |V'(\xi)|  \dd\xi,
\]
where $ \xi_0 \ge p^{\kappa}L^{-1} $. 
We thus have that
\[
\begin{aligned}
S(c,j) &\ll p^{r} (\log q)^\delta \Bigl( p^{(\gamma-\kappa)k-\kappa\ell} L^{\ell-k}+ p^{(\gamma-2\kappa)(\ell-1)}(Lp^{-\kappa})  \Bigr) \\
&\ll p^{r-(k+\ell)\kappa} (p^\gamma)^k  L^{\ell-k} (\log q)^{\delta},
\end{aligned}
\]
since 
\begin{equation}\label{condition}
p^{\gamma-2\kappa}\geq Lp^{-\kappa}.
\end{equation}

Going back to~\eqref{SumAspAdic}, we have that
\begin{equation*}
S(m,L) \ll p^{r+(1-k-\ell)\kappa} (p^\gamma)^k  L^{\ell-k} (\log q)^{\delta} + p^\kappa.
\end{equation*}
Since the estimate~\eqref{St bound} holds for all $1 \leq t \leq p^{\gamma-2\kappa}$, 
we know from~\cite[Eq.\,(16)]{MD} that its right-hand side is greater than $t^{1/2}$ throughout the same range. 
In particular, for $t=Lp^{-\kappa}$, we find that
\begin{equation*}
p^r\left(\frac{p^{\gamma-2\kappa}}{Lp^{-\kappa}}\right)^k 
(Lp^{-\kappa})^{\ell}(\log q)^{\delta} \geq (Lp^{-\kappa})^{1/2},
\end{equation*}
from which it follows that the first term dominates in our estimate of $S(m,L)$. 

Finally, we find that the case \(p=3\) requires only a slight modification of the above argument. In this case we split the \(l\)-sum into residue classes modulo
\(3^{\kappa+1}\), rather than modulo \(3^\kappa\), and write $l=c+3^{\kappa+1}t.$
Since $(1+3\mathbb Z_3)^3=1+9\mathbb Z_3,$
the factor \(1+3^{\kappa+1}c't\) has a unique cubic root in
\(1+3\mathbb Z_3\). The corresponding phase function
$$
f_c(t)=3s(mc)\left(1+3^{\kappa+1}c't\right)^{1/3}
$$
satisfies
$$
f_c\in\mathbf F\Bigl(\kappa+1,\frac23,\kappa,\infty,\infty,c',s(mc)c'\Bigr),
$$
since \(\iota(2/3)=1\) for \(p=3\). Moreover,
\bna
(1+3^{\kappa+1}c't)^{1/3}\equiv1\pmod3,
\ena
and hence the factor $\epsilon(v,3^\gamma)$ is constant with respect
to $t$ on each residue class.
Therefore the same argument applies
with \(3^{\gamma-2\kappa}\) replaced by
\(3^{\gamma-2\kappa-1}\), yielding the same bound up to an absolute
constant. Thus the conclusion remains valid for \(p=3\), after the
corresponding harmless adjustment of the lower bounds on \(\gamma\)
and \(L\).
\end{proof}

In the remainder of this section, we describe explicit $p$-adic exponent data and apply them to estimate the bound of $\mathcal{B}(L, M)$.
Using Lemma~\ref{Kloosterman sums bound} and applying the datum $ABA^3B(\omega_{01})$ 
with $\kappa=1$, $r=0$ and $p\not\in\{3,5,7, 11, 13,17,19\}$ yields the estimate
\begin{align}\label{upbound2}
\mathcal{B}(L, M)
\ll &\,\frac{X^{1+o(1)}}{q^2}p^{7/41}M\Bigl(\frac{q}{L}\Bigr)^{11/82}L^{57/82} \nonumber\\
\ll &\, \frac{X^{1+o(1)}}{q}p^{7/41}\frac{q^{175/82}}{XL^{36/82}}.
\end{align}

This bound~\eqref{upbound2} holds for
$$q \leq L^{36/175}X^{82/175-\eta},$$
with $\eta>0$ fixed.
Since the $A$- and $B$-processes produce $p$-adic exponent data that are effective for every prime $p$, 
the resulting bound also holds for $p\in\{3,5,7, 11, 13,17,19\}$ when $\gamma \geqslant n_0$, 
albeit with possibly different values of $r$ and $n_0$.
Moreover, after a suitable adjustment of the parameters, the same type of
bound applies for all sufficiently large $\gamma$. The finitely many
remaining values of $\gamma$ are treated directly as above.

\subsection{The case of moderate $L$}
We can apply Cauchy--Schwarz with $l$ inside the square and open the absolute value square. 
Therefore, $\mathcal{B}(L, M)$ is bounded by
\bea
\frac{XM^{1/2}}{q^2}\left(\sum_{l_1,l_2\geq 1}V_1\big(\frac{l_1}{L}\big)
\overline{V_1\big(\frac{l_2}{L}\big)}\sum_{m\geq 1}
\mathrm{Kl}_3(al_1 m;q)\overline{\mathrm{Kl}_3(al_2 m;q)}
V_2\big(\frac{m}{M}\big)\right)^{1/2}.\label{c-sum}
\eea
We consider two subcases.
\subsubsection{$l_1= l_2:=l$}
If $l_1= l_2=l$, the sum inside the parentheses above can be simply bounded by
\begin{equation}\label{l1=l2}
\begin{split}\sum_{l\geq 1}|V_1\big(\frac{l}{L}\big)|^2\sum_{m\geq 1}
|\mathrm{Kl}_3(al m;q)|^2V_2\big(\frac{m}{M}\big)\ll LM.
\end{split}\end{equation}
\subsubsection{$l_1 \neq l_2$}
If $l_1\neq l_2$, things get a little more complicated.
We apply the Poisson summation formula for the $m$-sum in \eqref{c-sum}, getting
\begin{equation*}\label{aa}
\sum_{m\geq 1}\mathrm{Kl}_3(al_1 m;q)
\overline{\mathrm{Kl}_3(al_2 m;q)}V_2\big(\frac{m}{M}\big)
=\frac{M}{q^{1/2}}\sum_{\tilde{m}\in \mathbb{Z}}
\mathcal{C}_{a}(\tilde{m},l_1,l_2;q)\widehat{V_2}\bigg(\frac{\tilde{m}}{q/M}\bigg)
\end{equation*}
where
\begin{equation*}\begin{split}\mathcal{C}_{a}(\tilde{m},l_1,l_2;q)&=\frac{1}{q^{1/2}}
\sum_{\gamma\in(\mathbb{Z}/q\mathbb{Z})}
\mathrm{Kl}_3(al_1 \gamma;q)\overline{\mathrm{Kl}_3(al_2 \gamma;q)}
e\bigg(\frac{\gamma \tilde{m}}{q}\bigg)
\\
&=\frac{1}{q^{3/2}}\,\,\sum_{x\bmod q\atop(x,q)=1}S(1,\overline{x}; q)
\overline{S}(1,\overline{(l_1\overline{l_2}x+\widetilde{m}\overline{al_2})}; q).
\end{split}
\end{equation*}
This type of sum has already been studied in \cite{DF1997}, and we quote the required estimate directly. 
\begin{lemma}
For $a \in \mathbb{Z}_p^{\times}$, $b \in \mathbb{Z}_p$ and $\gamma\geq1$, we have
\begin{equation*}
\sum_{x\bmod p^\gamma\atop(x,p)=1}S(1,x; p^\gamma)
 \overline{S}(1,ax\overline{(bx+1)}; p^\gamma)
 \ll p^{3\gamma/2}p^{(\min\{\gamma, \nu_p(a-1),\nu_p(b)\})/2}.
\end{equation*}
\end{lemma}
Substituting the above estimate into the rightmost curly bracket of \eqref{c-sum}, we obtain
\begin{align}\label{bb}
&\sum_{l_1,l_2\geq 1\atop l_1\neq l_2}V_1\big(\frac{l_1}{L}\big)
\overline{V_1\big(\frac{l_2}{L}\big)}\sum_{m\geq 1}
\mathrm{Kl}_3(al_1 m;q)\overline{\mathrm{Kl}_3(al_2 m;q)}
V_2\big(\frac{m}{M}\big)\nonumber\\\noalign{\vskip 2,8mm}
&=\frac{M}{q^{1/2}}\sum_{l_1,l_2\geq 1\atop l_1\neq l_2}V_1\big(\frac{l_1}{L}\big)
\overline{V_1\big(\frac{l_2}{L}\big)}\sum_{\tilde{m}\in \mathbb{Z}}
\mathcal{C}_{a}(\tilde{m},l_1,l_2;q)\widehat{V_2}\bigg(\frac{\tilde{m}}{q/M}\bigg)\nonumber\\\noalign{\vskip 2,8mm}
&\ll \frac{M}{q^{1/2}}\sum_{l_1,l_2\asymp L\atop l_1\neq l_2}\,
\sum_{\tilde{m}\ll q/M}
p^{(\min\{\gamma, \nu_p(l_1-l_2),\nu_p(\tilde{m})\})/2}\nonumber\\\noalign{\vskip 2,8mm}
&\ll \frac{M}{q^{1/2}}
\left(Lp^{\gamma/2}+\frac{L^2q}{M}\right)
\ll ML+L^2q^{1/2}.
\end{align}
Combining \eqref{l1=l2} and \eqref{bb},
we obtain 
\begin{align}\label{upbound3}
\mathcal{B}(L,M) \ll&\,\frac{X^{1+o(1)}M^{1/2}}{q^2}
\bigl(LM+L^2q^{1/2}\bigr)^{1/2}\nonumber\\
\ll& \,\frac{X^{1+o(1)}}q\bigl(\frac{1}{L}\frac{q^4}{X^2}
+L\frac{q^{3/2}}{X}\bigr)^{1/2}.
\end{align}
This bound \eqref{upbound3} is good as long as
$$q \leq L^{1/4}X^{1/2-\eta},$$ and $$q \leq X^{2/3-\eta}L^{-2/3}.$$

\subsection{Conclusion}

Let $L_0=X^{52/229}$ be the solution of the equation
$$L_0^{36/175}X^{82/175}=X^{2/3}L_0^{-2/3}=X^{1/2+7/458}.$$

We need to show that for any  small enough  $\eta>0$
and any prime power $q=p^\gamma$ satisfying
$$X^{1/2-\eta}\leq q\leq X^{1/2+7/458-\eta},$$ one has
\begin{equation*}
\begin{split}
\mathop{\sum_{n\geq 1}}_{n\equiv a\bmod q}
\lambda_{1\boxplus f}(n)&V\big(\frac{n}{X}\big)
-\frac{1}{\varphi(q)} \mathop{\sum_{n\geq 1}}_{(n,q)=1}
\lambda_{1\boxplus f}(n)V\big(\frac{n}{X}\big)\ll (X/q)^{1-\delta},
\end{split}
\end{equation*}
for some sufficiently small $\delta=\delta(\eta)>0$.

It is sufficient to show that this bound holds for any of the sums
\eqref{MNsums} for $L,M$ satisfying
$$1\leq LM\leq q^3/X.$$
\begin{itemize}
\item[-] If  $L\gg p^{-\kappa}q$, we use \eqref{MNbound2}.
\item[-] If  $L_0\ll L\ll p^{-\kappa}q$, we use \eqref{upbound2}.
\item[-] If  $q^4X^{-2+2\delta}\ll L\ll L_0$, we use \eqref{upbound3}.
\item[-] If $L\ll q^4X^{-2+2\delta}$, we use \eqref{upbound1}.
\end{itemize}

\bigskip

\noindent{\bf Acknowledgements.}
We gratefully acknowledge the many helpful suggestions 
of Professor Yongxiao Lin during the preparation of the paper.

\bibliographystyle{amsplain}

\end{document}